\documentclass[12pt]{amsart}

\usepackage[matrix,arrow,curve,frame]{xy}
\usepackage{amsmath,amsthm,amssymb,enumerate}
\usepackage{latexsym}
\usepackage{amscd}
\usepackage[colorlinks=true]{hyperref}
\usepackage{euscript}

\newtheorem*{thm*}{Theorem}
\newtheorem{thm}[subsection]{Theorem}
\newtheorem{defn}[subsection]{Definition}

\newtheorem{claim}[subsection]{Claim}

\newtheorem{remark}{Remark}

\theoremstyle{definition}

\newcommand{\Z}{\mathbb Z}
\newcommand{\F}{\mathbb F}

\DeclareMathOperator{\Sp}{\mathbb S}
\DeclareMathOperator{\GSp}{G\mathbb S}
\DeclareMathOperator{\BGSp}{BG\mathbb S}
\DeclareMathOperator{\BBGSp}{B^3G\mathbb S}

\DeclareMathOperator{\THH}{THH}
\DeclareMathOperator{\Hm}{H}

\DeclareMathOperator{\E}{E}

\DeclareMathOperator{\EM}{K}

\newfont{\german}{eufm10}

\newcommand\qu{/\kern-.7ex/}

\begin{document}
\pagestyle{plain}

\title
{Topological Hochschild homology of $\Hm(\Z/p^k)$.}

\author{Nitu Kitchloo}
\address{Department of Mathematics, Johns Hopkins University, Baltimore, USA}
\email{nitu@math.jhu.edu}
\thanks{Nitu Kitchloo is supported in part by the Simons Fellowship and the Max Planck Institute of Mathematics.}

\date{\today}


{\abstract

\noindent
In this short note we study the topological Hochschild homology of Eilenberg-MacLane spectra for finite cyclic groups. In particular, we show that the Eilenberg-MacLane spectrum $\Hm(\Z/p^k)$ is a Thom spectrum for any prime $p$ (except, possibly, when $p=k=2$) and we also compute its topological Hoschshild homology. This yields a short proof of the results obtained by Brun \cite{Br}, and Pirashvili \cite{P} except for the anomalous case $p=k=2$.}
\maketitle

\tableofcontents

\section{Introduction:}

\noindent
In this note we will describe the Eilenberg-MacLane spectrum $\Hm(\Z/p^k)$ as a $E_2$-Thom spectrum of a $p$-local spherical fibration over a space of the form $\Omega F(k)$, with $F(k)$ being a loop space. In order to do so, we will decompose $\Omega F(k)$ into a product $S^1 \times \Omega^2 S^3 \langle 3 \rangle$ (though not as loop spaces) so that the virtual bundle that induces the Thom spectrum can be realized as a product bundle. Over $S^1$, this Thom spectrum is equivalent to the Moore spectrum $M(p^k)$, and over $\Omega^2 S^3 \langle 3 \rangle$ the Thom spectrum is equivalent to the Eilenberg-MacLane spectrum $\Hm(\Z_{(p)})$. This realizes the decomposition $\Hm(\Z/p^k) = M(p^k) \wedge \Hm(\Z_{(p)})$.

\section{Acknowledgements}
\noindent
The author would like to acknowledge Achim Krause and Thomas Nikolaus for encouraging him to post this calculation and for also identifying some errors in an earlier draft. He would also like to thank John Lind for several helpful conversations. The author is grateful for the hospitality and support of the Simons Foundation and the Max Planck Institute of Mathematics, Bonn, where this work was completed. 

\section{Eilenberg-MacLane spectra as Thom spectra and THH.}

\noindent
In this section we will identify $\Hm(\Z/p^k)$ as a Thom spectrum (for any odd prime $p$ and arbitrary natural number $k$, or $p=2$ and $k \neq 2$) and compute is topological Hochschild homology. Towards this we begin with a general definition:

\medskip
\begin{defn}
For an arbitrary prime $p$ and $k \geq 1$, define $F(k)$ to be the homotopy pullback diagram of fibrations:
\[
\xymatrix{
\Omega S^3 \langle 3 \rangle \ar[r]^{=} \ar[d] & \Omega S^3 \langle 3 \rangle \ar[d] \\
F(k) \ar[d]^{\xi(k)} \ar[r]^{\iota} & \Omega S^3 \ar[d]^{\xi} \\
\EM(\Z,2) \ar[r]^{p^{k-1}} & \EM(\Z,2) .}
\]
Notice that the above pullback can be delooped by replacing $\xi$ by its delooping: $S^3 \longrightarrow \EM(\Z, 3)$. Hence we notice that $\iota$ is a loop map. Notice that on looping the above pullback once, both fibrations $\xi(k)$ and $\xi$ admit compatible splittings (but not as loop spaces).  
\end{defn}

\medskip
\noindent
Let $\GSp_{(p)}$ denote the $E_\infty$ monoid given by the units of the $p$-local sphere spectrum $\Sp_{(p)}$ \cite{ABGHR}. In other words $\GSp_{(p)}$ is defined as the subspace of components in $\Omega^{\infty}(\Sp_{(p)})$ which are invertible up to homotopy: 
\[ \pi_0 (\GSp_{(p)}) = \pi_0 (\Sp_{(p)}) ^{\times} = \Z_{(p)}^{\times}. \]
Let $M(p^k)$ denote the Moore spectrum given by the cofiber of the degree $p^k$-map on the sphere spectrum. $M(p^k)$ may be described as a Thom spectrum represented by any pointed map: 
\[ \tau : S^1 \longrightarrow \Z \times \BGSp_{(p)}, \]
with the property that the $\tau$ sends the generator of $\pi_1(S^1)$ to an element of the form: $1+ p^k \lambda \in \pi_1(\BGSp_{(p)}) = \Z_{(p)}^{\times}$, for any integer $\lambda$ prime to $p$.

\medskip
\begin{thm} \label{Thh}
Let $p$ be an odd prime. Then as an $E_2$-spectrum, $\Hm(\Z/p^k)$ has the structure of Thom spectrum over $\Omega F(k)$. In particular, using \cite{Bl} one obtains an equivalence of spectra: $\THH(\Hm(\Z/p^k)) = \Hm(\Z/p^k) \wedge F(k)_+$. Furthermore, the canonical map induced by $\Z/p^k \rightarrow \Z/p$: $\THH(\Hm(\Z/p^k)) \longrightarrow \THH(\Hm(\Z/p))$ is equivalent to the map induced by: 
\[ \iota : \Hm(\Z/p^k) \wedge F(k)_+ \longrightarrow \Hm(\Z/p) \wedge \Omega S^3_+. \]
\end{thm}
\begin{proof}
Consider the pullback diagram of fibrations obtained by looping the above diagram. Since $\Omega \xi$ admits a canonical splitting given by the suspension map: $S^1 \longrightarrow \Omega^2 S^3$, we have compatible splittings:
\[ \Omega \iota : \Omega F(k) = S^1 \times \Omega^2 S^3\langle 3 \rangle \longrightarrow S^1 \times \Omega^2 S^3\langle 3 \rangle = \Omega^2 S^3, \]
which is degree $p^{k-1}$ on the factor $S^1$, and the identity map on $\Omega^2 S^3 \langle 3 \rangle$. Let $\zeta$ denote the stable $p$-local spherical fibration over $\Omega^2 S^3$ represented by a double-loop map:
\[ \zeta : \Omega^2 S^3 \longrightarrow  \BGSp_{(p)}, \]
with the property that $\zeta$ restricts to the element $1+p \in \pi_1(\BGSp_{(p)}) = \Z_{(p)}^{\times}$. As mentioned above, the Thom spectrum of $\zeta$ restricted to $S^1$ under the map of degree $p^{k-1}$ is $Th(\zeta) = M(p^k)$. The restriction of $\zeta$ to $\Omega^2 S^3\langle 3 \rangle$ has a Thom spectrum equivalent to $\Hm(\Z_{(p)})$ \cite{Bl}. In particular, the Thom spectrum of the restriction of $\zeta$ to $\Omega F(k) = S^1 \times \Omega S^3 \langle 3 \rangle$ is equivalent to $M(p^k) \wedge \Hm(\Z_{(p)}) = \Hm(\Z/p^k)$. 
\end{proof}

\medskip
\noindent
As indicated earlier, the above theorem also has a variant for $p=2$. We will need to work slightly harder by first considering the $E_\infty$ monoid of units $\GSp_2$ for the 2-complete sphere $\Sp_2$, using that to deduce consequences for the $2$-local setting we are interested in. It is well known that $\pi_0(\GSp_2) = \Z_2^{\times} = \{ \pm 1\} \times \Z_2\langle 5 \rangle$, where $\Z_2 \langle 5 \rangle$ denotes the subgroup of units that are isomorphic to a copy of the 2-adic integers generated by the unit $5$. Let $\GSp^+_2$ denote the identity component of $\GSp_2$. 

\medskip
\noindent
Now consider the first two stages $P_2$ of the Postnikov decompositon for the third delooping of the units $\GSp_2$: 
\[
\xymatrix{
 \BBGSp_2^+ \ar[r]^{B^2(w_2)} \ar[d] & \EM(\Z/2, 4) \ar[d] \\
\BBGSp_2 \ar[d] \ar[r] & P_2 \ar[d] \\
\EM(\{\pm 1\}, 3) \times \EM(\Z_2\langle 5 \rangle, 3) \ar[r]^{=} & \EM(\{ \pm 1 \},3) \times \EM(\Z_2\langle 5 \rangle, 3).}
\]

\noindent
where $B^2(w_2)$ denotes a second delooping of the second Stiefel-Whitney class. It is not hard to calculate the k-invariant that defines $P_2$. This k-invariant $\theta$ is given by the projection $\pi$ onto the factor $\EM(\{ \pm 1\}, 3)$ followed by the Steenrod operation $Sq^2$:
\[ \theta = Sq^2 \pi : \EM(\{ \pm 1\},3) \times \EM(\Z_2\langle 5 \rangle, 3) \longrightarrow \EM(\{\pm 1\},3) \longrightarrow \EM(\Z/2, 5). \]

\noindent
As before, let us now consider the $\Sp_{(2)}$-bundle $\zeta$ on $\Omega^2 S^3$ obtained by taking double-loops on the element $3  \in \pi_3(\BBGSp_{(2)}) = \pi_0(\GSp_{(2)})$. The Thom-spectrum of the restriction of $\zeta$ to $S^1$ under the map of degree $2^{k-2}$ is $M(2^k)$ for $k>2$, or $M(2)$ if $k=2$. 

\medskip
\noindent
Pushing forward to the $2$-adic units, notice that the number $3$ can be expressed as a pair $(-1, \tau) \in \{\pm 1\} \times \Z_2 \langle 5 \rangle$ for some $\tau \in \Z_2 \langle 5 \rangle$. 
Using the Postnikov decomposition above, we see that the restriction of $\zeta$ to $\Omega^2 S^3 \langle 3 \rangle$ has a nonzero second Stiefel-Whitney class. By \cite{CMT}, we may conclude that the Thom spectrum for the restriction of $\zeta$ to $\Omega^2 S^3 \langle 3 \rangle$ is equivalent to $\Hm(\Z_{(2)})$. Proceeding as before, it follows that the restriction of $\zeta$ to $\Omega F(k)$ is $\Hm(\Z/2^k)$ for $k>2$ and $\Hm(\Z/2)$ for $k=2$, with the map $\THH(\Hm(\Z/2^k)) \longrightarrow \THH(\Hm(\Z/2))$ being equivalent to the map induced by $\iota$.  We therefore obtain:

\bigskip
\begin{thm} \label{Thh2}
Assume $k \neq 2$. Then as an $E_2$-spectrum, $\Hm(\Z/2^k)$ has the structure of Thom spectrum over $\Omega F(k)$. In particular, invoking results from \cite{Bl}, one obtains an equivalence of spectra: $\THH(\Hm(\Z/2^k)) = \Hm(\Z/2^k) \wedge F(k)_+$. Furthermore, the canonical map induced by $\Z/2^k \rightarrow \Z/2$: $\THH(\Hm(\Z/2^k)) \longrightarrow \THH(\Hm(\Z/2))$ is equivalent to the map induced by: 
\[ \iota : \Hm(\Z/2^k) \wedge F(k)_+ \longrightarrow \Hm(\Z/2) \wedge \Omega S^3_+. \]
\end{thm}

\bigskip
\noindent
The next step is to compute the homotopy of $\THH(\Hm(\Z/p^k))$. This reduces to the calculation of the mod $p$-cohomology of $F(k)$, while keeping track of  higher Bocksteins.

\medskip
\begin{claim} Let $p$ be an arbitrary prime, and $k>1$. Then as a Hopf algebra, $\Hm^\ast(F(k), \Z/p)$ is: 
\[ \Hm^\ast(F(k), \Z/p) = \E(x_{2p-1}) \otimes \Gamma (x_{2p}) \otimes \F_p[x_2], \]
where subscripts denote the degrees of the respective classes. For $k=1$, we have:
\[ \Hm^\ast(\Omega S^3, \Z/p) =  \Gamma (y_2). \]
Furthermore, the map $\iota : F(k) \longrightarrow \Omega S^3$ has the property:
\[ \iota^\ast \gamma_{pn}(y_2) = \gamma_n(x_{2p}), \quad \iota^\ast \gamma_{m}(y_2) = 0, \quad \mbox{if $p$ does not divide $m$}. \]
\end{claim}
\begin{proof}
Consider the three-by-three diagram of fibrations:
\[
\xymatrix{ \ast \ar[r] \ar[d] & \Omega S^3\langle 3 \rangle \ar[r]^{=} \ar[d] & \Omega S^3 \langle 3 \rangle \ar[d] \\
\EM(\Z/p^{k-1}, 1) \ar[r] \ar[d]^{=} & F(k) \ar[d]^{\xi(k)} \ar[r]^{\iota} & \Omega S^3 \ar[d]^{\xi} \\
\EM(\Z/p^{k-1}, 1) \ar[r]^{\quad \beta} & \EM(\Z,2) \ar[r]^{p^{k-1}} & \EM(\Z,2) .}
\]
One may analyze the Serre spectral sequence in cohomology with for the two fibrations $\iota$ and $\xi(k)$. It is easy to see that given $k>1$, the spectral sequence for $\xi(k)$ collapses (both integrally and over $\Z/p$). The spectral sequence with coefficients in $\Z/p$ for $\iota$ has one differential $d_2$ with target $\gamma_1(y_2)$ which wipes out all classes in $\Hm^*(\Omega S^3, \Z/p)$ generated by $\gamma_m(y_2)$ for which $p$ does not divide $m$. Hence $\iota^*$ and $\xi(k)^*$ yield an extension of Hopf algebras:
\[ 1 \longrightarrow \Gamma(x_{2p}) \otimes \F_p[x_2] \longrightarrow \Hm^*(F(k), \Z/p) \longrightarrow \E(x_{2p-1}) \longrightarrow 1. \]
First note that the (unique) lift of $x_{2p-1}$ is primitive by degree reasons. If $p$ is odd, then we know the square of $x_{2p-1}$ is trivial. If $p=2$, then the square of $x_{2p-1}$ is also primitive but since there are no primitives in degree $4p-2$, it follows that this class squares to zero even if $p=2$. Hence the above extension splits as Hopf algebras. 
\end{proof}

\medskip
\begin{remark}
Recall that for $k>1$, the Serre spectral sequence for the fibration $\xi(k)$ collapses with coefficients in $\Z$ and $\Z/p$. It easily follows that all additive extensions are trivial in the integral spectral sequence for the bundle $\xi(k)$. In other words, we have an isomorphism of groups: $\Hm^*(F(k),\Z) = \Hm^*(\Omega S^3\langle 3 \rangle, \Z) \otimes \Z[x_2]$ for $k>1$. Alternatively stated, the class $x_{2p-1} \gamma_{p^{n-1}-1}(x_{2p})$ supports a non-trivial Bockstein homomorphism of height $n$ with target $\gamma_{p^{n-1}}(x_{2p})$. From this observation, it is straightforward to recover the results by Brun \cite{Br} and Pirashvili \cite{P}. 
\end{remark}

\medskip
\begin{remark}
Even though $F(k)$ is an H-space, it is not known if the equivalence in Theorem \ref{Thh} respects the algebra structure. In particular, we can only conclude that $\Hm_*(F(k), \Z/p^k)$ is equivalent to $\THH_*(\Hm(\Z/p^k))$ as groups. 
\end{remark}

\medskip
\begin{remark}
There appears to be some confusion in the literature regarding the possibility of realizing $\Hm(\Z/p^k)$ as a Thom spectrum for $k>1$. In \cite{BCS} it is claimed that such a realization does not exist (see remark following Theorem 1.4) citing \cite{Bl} as a reference for more details. This appears to contradict our result. However, \cite{Bl} (Remark 9.4) only claims that such a spectrum cannot be constructed from a virtual bundle over $\Omega^2 S^3$. This claim is supported by studying the action of the Dyer-Lashof operaton $Q_2$. Notice that even though $\Omega F(k)$ is abstractly equivalent to $\Omega^2 S^3$ as spaces, the loop structure on $\Omega F(k)$ is different from $\Omega^2 S^3$ if $k>1$. We suspect that this difference is witnessed by $Q_2$, resolving the apparent contradiction. 
\end{remark}

\pagestyle{empty}
\bibliographystyle{amsplain}

\begin{thebibliography}{10}
\bibitem[ABGHR]{ABGHR}, M. Ando, A. J. Blumberg, D. Gepner, M. J. Hopkins, C. Rezk, \textit{Units of ring spectra, orientations, and Thom spectra via rigid infinite loop space theory}, available at: arXiv:1403.4320v1, 2014. 
\bibitem[Bl]{Bl} A. Blumberg, \textit{$\THH$ of Thom spectra which are $E_\infty$ ring spectra}, Journal of Topology, 3, 2010, 535--560. 
\bibitem[BCS]{BCS} A. Blumberg, R. Cohen, C. Schlichtkrull, \textit{Topological Hochschild homology of Thom spectra and free loop space}, Geometry and Topology, 14, 2010, 1165--1242. 
\bibitem[Br]{Br} M. Brun, \textit{Topological Hochschild homology of $\Z/p^n$}, Journal of Pure and Applied Alg., 148, 2000, 29--76. 
\bibitem[CMT]{CMT} F. R. Cohen, J. P. May, L. R. Taylor, \textit{$\EM(\Z,0)$ and $\EM(\Z/2, 0)$ as Thom spectra}, Illinois J. of Math., 25(1), 1981, 99--106.  
\bibitem[P]{P} T. Pirashvili, \textit{On the Topological Hochschild homology of $Z/p^kZ$}, Communications in Algebra, 23(4), 1995, 1545--1549. 

\end{thebibliography}
\providecommand{\bysame}{\leavevmode\hbox
to3em{\hrulefill}\thinspace}

\end{document}